\documentclass[a4paper,reqno,11pt,twoside]{amsart}

\usepackage{graphicx}
\usepackage{mathrsfs}
\usepackage{amssymb,amsmath, amsfonts, amsthm,soul}
\usepackage{latexsym}
\usepackage{color} %{\color{red}..............OK}
\usepackage[dvips]{epsfig}
\usepackage{cite,enumitem,graphicx}
\usepackage[colorlinks=true,urlcolor=blue,
citecolor=red,linkcolor=blue,linktocpage,pdfpagelabels,
bookmarksnumbered,bookmarksopen]{hyperref}
\usepackage[english]{babel}
\usepackage[left=2.61cm,right=2.61cm,top=2.72cm,bottom=2.72cm]{geometry}
\usepackage[hyperpageref]{backref}
\allowdisplaybreaks[1]

\numberwithin{equation}{section}
\numberwithin{equation}{section}
\newtheorem{theorem}{Theorem}[section]

\newtheorem{proposition}[theorem]{Proposition}
\newtheorem{lemma}[theorem]{Lemma}

\newcommand{\R}{\mathbb{R}}

\DeclareMathOperator{\supp}{supp}
\renewcommand{\le}{\leqslant}
\renewcommand{\ge}{\geqslant}

\renewcommand{\l }{\lambda}
\newcommand{\n }{\nabla }

\renewcommand{\H}{H^1_0}

\newcommand{\into }{\int_{\Omega}}

\def\bbm[#1]{\mbox{\boldmath $#1$}}
\newcommand{\beq }{\begin{equation}}
\newcommand{\eeq }{\end{equation}}

 \renewcommand{\(}{\left(}
\renewcommand{\)}{\right)}

\newcommand{\e}{\varepsilon}
\title[Generalized SN system]{Generalized Schr\"odinger-Newton system\\ in dimension $N\ge 3$: critical case}

\author[A. Azzollini]{Antonio Azzollini}
\author[P. d'Avenia]{Pietro d'Avenia}
\author[G. Vaira]{Giusi Vaira}

\address[A. Azzollini]{\newline\indent
Dipartimento di Matematica, Informatica ed Economia \newline\indent
Universit\`a degli Studi della Basilicata\newline\indent
Via dell'Ateneo Lucano 10, 85100 Potenza, Italy}
\email{\href{mailto:antonio.azzollini@unibas.it}{antonio.azzollini@unibas.it}}

\address[P. d'Avenia]{\newline\indent
	Dipartimento di Meccanica, Matematica e Management
	\newline\indent
	Politecnico di Bari
	\newline\indent
	Via Orabona 4,  70125  Bari, Italy}
\email{\href{mailto:pietro.davenia@poliba.it}{pietro.davenia@poliba.it}}

\address[G. Vaira]{\newline\indent
	Dipartimento di Matematica ``G. Castelnuovo''
	\newline\indent
	Universit\`{a} degli Studi di Roma Sapienza
	\newline\indent
	Piazzale Aldo Moro 5, 00185 Roma, Italy}
\email{\href{mailto:giusi.vaira@uniroma1.it}{giusi.vaira@uniroma1.it}}

\thanks{A. Azzollini, P. d'Avenia, and G. Vaira are partially supported by a grant of the group GNAMPA of INdAM. A. Azzollini is also partially supported by Fondo R.I.L. 2015, Project ``Studio di equazioni differenziali alle derivate parziali nonlineari'', Universit\`a degli Studi della Basilicata. G. Vaira is also partially supported by MIUR-PRIN project-201274FYK7 005.}

\begin{document}
\begin{abstract}
	In this paper we study a system which is equivalent to a nonlocal version of the well known Brezis Nirenberg problem. The difficulties related with the lack of compactness are here emphasized by the nonlocal nature of the critical nonlinear term.\\
	We prove existence and nonexistence results of positive solutions when $N=3$ and existence of solutions in both the resonance and the nonresonance case for higher dimensions.
\end{abstract}

\maketitle

\begin{center}
	\begin{minipage}{10.3cm}
		\small
		\tableofcontents
	\end{minipage}
\end{center}

\section{Introduction and statement of the main results}
In this paper we consider the following Schr\"odinger-Newton type system
\begin{equation}\label{pb}
\tag{$\mathcal{SN}$}
	\begin{cases}
	-\Delta u = \lambda u + |u|^{2^*-3}u\phi
	& \hbox{in } \Omega\\
	-\Delta\phi= |u|^{2^*-1}
	& \hbox{in } \Omega\\
	u=\phi=0
	&\hbox{on } \partial\Omega
	\end{cases}
\end{equation}
where $\Omega\subset \mathbb R^N$, $N\ge 3$, is a smooth and bounded domain and $2^*=\frac{2N}{N-2}$ is the critical Sobolev exponent. \\
By applying a standard reduction, it is easy to see that \eqref{pb} is equivalent to the following critical version of a Choquard type equation in a bounded domain, with Dirichlet boundary conditions
\begin{equation}\label{pb2}
	\begin{cases}
	-\Delta u = \lambda u + \left(\displaystyle\int_{\Omega}G(x-y)|u(y)|^{2^*-1}\,dy\right)|u|^{2^*-3}u
	& \hbox{in } \Omega,\\
	u=0
	&\hbox{on } \partial\Omega,
	\end{cases}
\end{equation}
where $G$ is the Green function of the Laplacian in the domain $\Omega$ with Dirichlet homogeneous boundary conditions (see e.g. \cite{Evans}). As emphasized in \cite{AzzDav}, where \eqref{pb} has been firstly introduced and positive solutions have been considered, the problem we are going to study is variational, and the particular choice of the powers in the first and the second equations brings as a consequence difficulties in questions related with compactness.\\
The analogy arising by a comparison with the classic problem of finding solutions to
  \begin{equation}\label{eq:BN}
	\begin{cases}
	-\Delta u = \lambda u + |u|^{2^*-2}u
	& \hbox{in } \Omega,\\
	u=0
	&\hbox{on } \partial\Omega,
	\end{cases}
\end{equation}
leads to interpret \eqref{pb2} as a nonlocal version of the well known Brezis-Nirenberg problem, suggesting approaches similar to those in \cite{BN} and in \cite{CFP,GR}.\\

As observed by Brezis and Nirenberg in \cite{BN}, the presence of a nonlinearity critically growing causes problems in obtaining compactness  for the Palais-Smale sequences unless they lie on a certain sublevel of the functional associated to the problem. As a consequence, it is quite clear that large values of $\lambda$ give a positive contribution in gaining compactness.\\
On the other hand, it was also observed that, if $\lambda$ is greater than the first eigenvalue of the Laplacian with zero Dirichlet condition on $\Omega$, we have no positive solution to \eqref{eq:BN}. The question related to the existence of sign-changing solutions of \eqref{eq:BN} for $\lambda \ge \lambda_1$ was handled and solved in dimension $N\ge 4$ in \cite{CFP}, and successively reconsidered in \cite{GR} in presence of a general nonlinearity $g(x,u)$ in the place of the linear term $\lambda u$.

In accordance with what is expected, the dimension of the space  where we set the problem and the position of $\lambda$ with respect to the eigenvalues in the spectrum of the Laplacian play a very important role also when we try to solve \eqref{pb}. In particular, while for $\lambda$ not belonging to the spectrum and $N\ge 4$ suitable estimates on the functional permit to achieve results similar to those already known for the Brezis-Nirenberg problem, the three-dimensional case  and the so called  ``resonance case'' are quite more delicate.\\

In order to present our results, let us first introduce some notations. Let $\lambda_k$'s, $k\in\mathbb N$, be the eigenvalues of $-\Delta$ with homogeneous Dirichlet boundary conditions on $\Omega$. It is well known that $$0<\lambda_1<\lambda_2<\ldots\ldots<\lambda_k<\ldots$$ with $\lambda_k \to+\infty$ as $k\to\infty$. We denote by $\sigma(-\Delta)$ the spectrum of $-\Delta$.
\\

We recall the following results whose proofs are essentially contained in \cite{AzzDav}.
	\begin{proposition}\label{pr:ad}
	Problem \eqref{pb} has no positive solution for $\lambda\ge \lambda_1$ and, if $\Omega$ is starshaped, no solution for $\lambda\le 0$.\\
	Moreover, if $N=3$ and $\Omega$ corresponds to the ball $B_R$ centered in 0 with radius $R>0$, the problem \eqref{pb} has at least a positive radial ground state solution for any $\lambda \in \big]\frac 3{10}\lambda_1,\lambda_1\big[$.
	\end{proposition}

We are able to improve results in \cite{AzzDav} as follows

\begin{theorem}\label{Th1}
Assume $N=3$ and $\Omega$ corresponding to the ball $B_R$ centered in 0 with radius $R>0$. Then
\begin{enumerate}[label=(\roman*),ref=\roman*]
\item\label{i12} problem \eqref{pb} has a positive radial ground state solution for any $\lambda\in\, \big] \left(\frac 1 4 + \frac{2}{5\pi^2}\right)\lambda_1,\lambda_1\big[$;
\item\label{ii12} problem \eqref{pb} has no positive solution for $\lambda\in ]0,\lambda^* ]$, where $\lambda^*\in \big]\frac{\lambda_1}{16},\frac{9}{64}\lambda_1\big[$ is explicitly determined as a solution of a suitable equation.
\end{enumerate}
\end{theorem}
We specify that when we say {\it positive solution} we mean a couple $(u,\phi)$ where both the functions are positive in $\Omega$.

The existence result has been obtained by means of refined estimates on the explicit expression of the solution $\phi$ of the second equation, when $u$ in chosen inside the well known one parameter family of functions introduced by Brezis and Nirenberg in \cite{BN}. We recall that the components of such a family, obtained cutting off the solutions of the critical problem \eqref{eq:BN} set in $\Omega=\R^N$, are used by Brezis and Nirenberg as test functions to prove that there exists at least a Palais-Smale sequence lying in the compactness sublevel.\\
As regards the nonexistence result, we are going to use an argument exploited in \cite{BN} for \eqref{eq:BN}, and based on the a priori information about the radial symmetry of any positive solution of \eqref{eq:BN}. To this end, we preliminarily prove an analogous result on the symmetry of all positive solutions of \eqref{pb} by showing that, since the general assumptions of the maximum principle hold, the moving plane method applies in a quite natural way to the solutions of our system.\\

Higher dimensions permit to obtain more convenient estimates by which not only we enlarge the range of existence for what concerns positive solutions, but we are also able to study the presence of sign changing solutions for $\lambda\ge\lambda_1$. The situation changes considerably depending wheter  $\lambda$ is in the spectrum or not.

We have the following two results.

\begin{theorem}\label{Th2}
	Assume $N\ge 4$. Then problem \eqref{pb} has
	\begin{enumerate}[label=(\roman*),ref=\roman*]
		\item\label{i13}  a positive ground state solution for any $\lambda\in ]0, \lambda_1[$;
		\item\label{ii13} a sign changing solution for every $\lambda\in ]\lambda_k,\lambda_{k+1}[$.
	\end{enumerate}
\end{theorem}

\begin{theorem}\label{Th3}
	Assume $N\ge 6$ and $\lambda \in \sigma(-\Delta)$. Then \eqref{pb} possesses a sign changing solution.
\end{theorem}

Comparing these results with those in \cite{GR}, we observe the loss of the dimension $N=5$ when we are in the  resonance case. The differences between the resonance and the nonresonance case, already pointed out by Gazzola and Ruf for the Brezis Nirenberg problem, here are emphasized by the behaviour of the nonlocal nonlinear term when we compute the functional on finite dimensional subspaces of $\H(\Omega)$. In particular, it is quite interesting to observe that, differently from what happens for the critical term in the functional associated to the problem \eqref{eq:BN}, in finite dimensional subspaces the integral of the nonlocal term does not have the growth of the norm to the critical power, but it behaves as the norm to the power $2(2^*-1)$.

The paper is so organized: in Section \ref{preliminari}, after having recalled the reduction method which classically applies to this kind of problems, we present some useful properties related with the reduced functional, and in particular we study the compactness of its Palais-Smale sequences. \\
In Section \ref{N3} we are interested in positive solutions when $N=3$ and we prove the existence and the nonexistence results contained in Theorem \ref{Th1}.\\
Finally in Section \ref{N4} we consider the case $N\ge 4$  looking for positive and sign changing solutions. In particular for the existence of the latter we take advantage of the Linking Theorem in \cite{Rab} which we apply both in the resonance and in the nonresonance case in order to prove Theorem \ref{Th2} and Theorem \ref{Th3}.
\\\\
{\bf{ Notations:}} In what follows we let $H^1_0(\Omega)$ the usual Sobolev space equipped with the norm $\|u\| :=(\int_\Omega |\nabla u|^2\, dx)^\frac{1}{2}$ and, for any $u\in L^q(\Omega)$ we let $|u|_q:=\left(\int_\Omega |u|^q\, dx\right)^{\frac 1 q}$. Moreover, with $C,C_i$ we denote positive constants that can vary also from line to line and  by $B_r$ the ball in $\R^N$ centered at zero with radius $r$.

\section{Preliminaries}\label{preliminari}
\subsection{The reduction method}
As it is classical in the study of this type of systems, for every $u\in H^1_0(\Omega)$ there exists a unique $\phi_u\in H^1_0(\Omega)$ that solves the second equation of \eqref{pb}. Hence we can reduce the system \eqref{pb} to the boundary value problem
\begin{equation}\label{pbrid}
\begin{cases}
-\Delta u = \lambda u + |u|^{2^*-3}u\phi_u
& \hbox{in } \Omega,\\
u=0
&\hbox{on } \partial\Omega.
\end{cases}
\end{equation}
In fact it can be easily proved that $(u,\phi)\in H^1_0(\Omega)\times H^1_0(\Omega)$ is a solution of \eqref{pb} if and only if $u$ solves \eqref{pbrid} and $\phi=\phi_u$.\\
Moreover, solutions of \eqref{pbrid} can be found as critical points of the $C^1$ one-variable functional  $I:H^1_0(\Omega)\to \mathbb R$
\begin{equation*}%\label{funz}
\begin{split}
I(u)
&=
\frac 12 \int_\Omega|\nabla u|^2\, dx -\frac{\lambda}{2}\int_\Omega |u|^2\, dx -\frac{1}{2(2^*-1)}\int_\Omega |\nabla\phi_u|^2\, dx\\
&=
\frac 12 \int_\Omega|\nabla u|^2\, dx -\frac{\lambda}{2}\int_\Omega |u|^2\, dx -\frac{1}{2(2^*-1)}\int_\Omega \phi_u |u|^{2^*-1}\, dx
\end{split}
\end{equation*}
since for all $u,v\in H^1_0(\Omega)$ we have
\[
I'(u)[v]
=\int_\Omega \nabla u \nabla v \, dx
- \lambda \int_\Omega u  v \, dx
- \int_\Omega |u|^{2^*-3}u\phi_u v \, dx.
\]

In the next lemma we summarize the properties of such a function $\phi_u$ that will be useful in the following.
\begin{lemma}\label{Lemmaphi}
For every fixed $u\in H^1_0(\Omega)$ we have:
\begin{enumerate}[label=(\roman*),ref=\roman*]
\item \label{pos} $\phi_u \ge 0$ a.e. in $\Omega$;
\item \label{rescaling} for all $t>0$, $\phi_{tu}=t^{2^*-1}\phi_u$;
\item \label{phiuu2*-1} $\|\phi_u\|\le S^{-\frac{2^*}{2}}  \|u\|^{2^*-1}$, where $S=\inf_{v\in H^1(\R^N)\setminus\{0\}}\|\n v\|_2^2/\|v\|^2_{2^*}$;
\item \label{13} $\|\phi_u\|^2\ge 2\delta |u|_{2^*}^{2^*} - \delta^2\|u\|^2$ for any $\delta>0$.
%\item \label{13bis} $\|\phi_u\|^2\ge 2 |u|_{2^*}^{2^*} - \|u\|^2$.
\end{enumerate}
Moreover
\begin{enumerate}[label=(\roman*),ref=\roman*]
\setcounter{enumi}{4}
\item \label{iv} for every $u, v \in H^1_0(\Omega)$,
$$\int_\Omega \phi_u |v|^{2^*-1}\, dx = \int_\Omega \phi_v |u|^{2^*-1}\, dx;$$
\item \label{vi} for every $u, u_1,\ldots,u_k\in H^1_0(\Omega)$,
$$
\left|\phi_u-\sum_{i=1}^{k}\phi_{u_i}\right|_{2^*}
\le
\frac{1}{S}\left||u|^{2^*-1}-\sum_{i=1}^{k}|u_i|^{2^*-1}\right|_{\frac{2^*}{2^*-1}};$$
\item \label{vii} if $(u_n)$ in $H^1_0(\Omega)$ and $u\in\H(\Omega)$ are such that $u_n \rightharpoonup u$  in $H^1_0(\Omega)$, then, up to subsequences, $\phi_{u_n}\rightharpoonup \phi_u$ in $H^1_0(\Omega)$ and strongly in $L^p(\Omega)$ for all $p\in[1,2^*)$. Moreover
\begin{equation}\label{starstarstarp}
\int_\Omega  \phi_{u_n} |u_n|^{2^*-1}\, dx-\int_\Omega \phi_{u_n-u}|u_n-u|^{2^*-1}\, dx = \int_\Omega \phi_u |u|^{2^*-1}\, dx+ o_n(1);
\end{equation}
\item \label{viii} if $W\subset\H(\Omega)$ is a finite dimensional subspace, then there exists $C=C(W)>0$ such that for any $w\in W$ we have
	\begin{equation*}%\label{eq:equiv}
		C^{-1} \|w\|^{2^*-1}\le \|\phi_w\| \le C  \|w\|^{2^*-1}.
	\end{equation*}
\end{enumerate}
\end{lemma}

\begin{proof}
Property (\ref{pos}) is trivial (see e.g. \cite{DM04}) and (\ref{rescaling}) easily follows from
\[
-\Delta \phi_{tu}=t^{2^*-1}|u|^{2^*-1}= -\Delta(t^{2^*-1}\phi_u) \quad \hbox{in }\Omega
\]
and $\phi_{tu}=\phi_u=0$ on $\partial\Omega$.\\
Multiplying the second equation of \eqref{pb} by $\phi_u$, integrating and using H\"older and Sobolev inequalities we have
\[
\|\phi_u\|^2
= \int_{\Omega} \phi_u |u|^{2^*-1}\, dx
\le |\phi_u|_{2^*} |u|_{2^*}^{2^*-1}
%\le \frac{1}{S^{\frac 12}} \|\phi_u\| |u|_{2^*}^{2^*-1}
\le S^{-\frac{2^*}{2}}\|u\|^{2^*-1}\|\phi_u\|
\]
and then (\ref{phiuu2*-1}).\\
Moreover, multiplying the second equation of \eqref{pb} by $|u|$ and integrating we have
\[
|u|_{2^*}^{2^*} = \int_{\Omega} \n \phi_u \n |u| \, dx \le \frac{1}{2\delta}  \|\phi_u\|^2 + \frac{\delta}{2} \| u\|^2
\quad
\hbox{for any }\delta>0
\]
and so (\ref{13}).\\
To obtain \eqref{iv} we observe that
$$\int_\Omega \phi_v |u|^{2^*-1}\, dx = \int_\Omega \nabla\phi_u\nabla\phi_v\, dx =\int_\Omega \phi_u|v|^{2^*-1}\, dx.$$
A further simple computation gives
\begin{equation*}
\begin{aligned}
\left|\phi_u-\sum_{i=1}^{k}\phi_{u_i}\right|_{2^*}^2
&\le \frac{1}{S} \left\|\phi_u-\sum_{i=1}^{k}\phi_{u_i}\right\|^2\\
&=\frac{1}{S} \int (\phi_u-\sum_{i=1}^{k}\phi_{u_i})(|u|^{2^*-1}-\sum_{i=1}^{k}|u_i|^{2^*-1})\,dx\\
&\le \frac{1}{S} \left|\phi_u-\sum_{i=1}^{k}\phi_{u_i}\right|_{2^*}\left||u|^{2^*-1}-\sum_{i=1}^{k}|u_i|^{2^*-1}\right|_{\frac{2^*}{2^*-1}}
\end{aligned}
\end{equation*}
and \eqref{vi} follows.\\
To prove the first part of (\ref{vii}) we can proceed as in \cite[Proposition 2.4]{Ruiz}.
%, we observe that, since, by assumptions, the sequence $(u_n)$ is bounded then by \eqref{phiuu2*-1} also $(\phi_{u_n})$ is bounded in $H^1_0(\Omega)$. Then $\phi_{u_n} \rightharpoonup v$ in $H^1_0(\Omega)$. Hence for any $\varphi\in C^\infty_0(\Omega)$
%\begin{equation}
%\label{le211}
%\int_\Omega |u_n|^{2^*-1}\varphi\, dx = \int_\Omega \nabla\phi_{u_n}\nabla \varphi\, dx \longrightarrow \int_\Omega \nabla v \nabla\varphi\, dx.
%\end{equation}
%Moreover the sequence $(|u_n|^{2^*-1})$ is bounded in $L^{\frac{2^*}{2^*-1}}(\Omega)$ and $|u_n|^{2^*-1}\to |u|^{2^*-1}$ a.e. in $\Omega$. Hence (see e.g. \cite[Proposition 5.4.7]{W2013}) $|u_n|^{2^*-1}\rightharpoonup |u|^{2^*-1}$ in $L^{\frac{2^*}{2^*-1}}(\Omega)$ and then for all $\varphi\in C^\infty_0(\Omega)$
%\begin{equation}
%\label{le212}
%\int_\Omega |u_n|^{2^*-1}\varphi\, dx \longrightarrow \int_\Omega |u|^{2^*-1}\varphi\, dx =\int_\Omega \nabla\phi_u\nabla\varphi\, dx.
%\end{equation}
%Then by \eqref{le211} and \eqref{le212} we get
%%$$\int_\Omega \nabla v \nabla\varphi\, dx = \int_\Omega \nabla\phi_u \nabla\varphi\, dx$$ and hence
%$\phi_u=v$.\\
Furthermore by applying \eqref{iv} we get
\begin{align*}
&\int_\Omega \phi_{u_n}|u_n|^{2^*-1}\, dx -\int_\Omega \phi_{u_n-u}|u_n-u|^{2^*-1}\, dx \\
&=\int_{\Omega} \left(\phi_{u_n}-\phi_{u_n-u}\right)\left(|u_n|^{2^*-1}-|u_n-u|^{2^*-1}\right)\, dx
+2\int_\Omega \left(\phi_{u_n}-\phi_{u_n-u}\right)|u_n-u|^{2^*-1}\, dx.
\end{align*}
%\begin{equation}\label{PS6}
%\begin{split}
%\end{split}
%\end{equation}
An easy variant of the classical Brezis-Lieb Lemma (see also \cite[Lemma 2.5]{MVS}) yields that
$$
|u_n|^{2^*-1}-|u_n-u|^{2^*-1} \to |u|^{2^*-1}\quad \hbox{in } L^{\frac{2^*}{2^*-1}}(\Omega) \hbox{ as }n\to+\infty
$$ and applying \eqref{vi} we get that
\begin{equation}\label{starstarp}
\phi_{u_n}-\phi_{u_n-u}\to \phi_u \quad \hbox{in } L^{2^*}(\Omega) \hbox{ as }n\to+\infty.
\end{equation}
Hence
\[
\int_{\Omega} \left(\phi_{u_n}-\phi_{u_n-u}\right)\left(|u_n|^{2^*-1}-|u_n-u|^{2^*-1}\right)\, dx
\to \int_\Omega \phi_u |u|^{2^*-1}\, dx
\hbox{ as }n\to+\infty.
\]
Moreover, applying again \cite[Proposition 5.4.7]{W2013}, we have $|u_n-u|^{2^*-1}\rightharpoonup 0$ in $L^{\frac{2^*}{2^*-1}}(\Omega)$.
Hence, since $\phi_u\in L^{2^*}(\Omega)$ and using also \eqref{starstarp},

	\begin{multline*}
		\int_\Omega \left(\phi_{u_n}-\phi_{u_n-u}\right)|u_n-u|^{2^*-1}\, dx\\
		= \int_\Omega \left(\phi_{u_n}-\phi_{u_n-u}-\phi_u\right) |u_n-u|^{2^*-1}\, dx
			+ \int_\Omega \phi_u |u_n-u|^{2^*-1}\, dx \to 0
	\end{multline*}
as $n\to +\infty$.\\
Finally, in order to see (\ref{viii}), consider $W$ a finite dimensional subspace of $\H(\Omega)$.  The second inequality is already known by (\ref{phiuu2*-1}) and it actually does not depend on $W$.\\
Now, pick any $w\in W$ and consider the second equation of \eqref{pb}.
\\
Multiplying by $|w|$, integrating and applying the Holder inequality, we have

	\begin{equation*}
		|w|_{2^*}^{2^*}=\int_{\Omega} \n \phi_w \n |w|\, dx\le \|\phi_w\|\|w\|.
	\end{equation*}
By equivalence of norms in finite dimensional spaces, we have that there exists $C$ depending on $W$ such that $C\|w\|\le |w|_{2^*}$, and then we are done.
\end{proof}

We conclude this section observing that, applying (\ref{13}) in Lemma \ref{Lemmaphi} for $\delta=1$, we have that for all $u\in H^1_0(\Omega)$
\begin{equation}
\label{J}
I(u)\le \frac{N}{N+2}\| u\|^2-\frac{\lambda}{2}|u|^2_2-\frac{N-2}{N+2}|u|_{2^*}^{2^*}.
\end{equation}

\subsection{The Palais-Smale condition}
As usual, in order to have compactness, a first crucial step consists in checking the Palais-Smale (PS for short) condition. We recall that a sequence $(u_n)\subset H^1_0(\Omega)$ is called a PS sequence for $I$ at level $c$ if $I(u_n)\to c$ and $I'(u_n) \to 0 $ in $[H^1_0(\Omega)]^{-1}$. The functional $I$ satisfies the PS condition at level $c$, if every PS sequence at level $c$ has a convergent subsequence in $H^1_0(\Omega)$.\\
Since we are in the critical case, we do not know if the functional $I$ satisfies the PS condition at all levels.  However there is a set of values in which it is preserved.
\begin{lemma}\label{PS}
	The functional $I$ satisfies the Palais-Smale condition in $\left]-\infty, \frac{2}{N+2}S^{\frac N 2}\right[$.
\end{lemma}
\begin{proof}
	Let $(u_n)\subset H^1_0(\Omega)$ be a PS sequence for $I$  at level $c<  \frac{2}{N+2}S^{\frac N 2}$. First we show that $(u_n)$ is bounded. In what follows all the convergences are meant up to a subsequence.\\
	Indeed, using also \eqref{13} of Lemma \ref{Lemmaphi}, we have
	\[
	2c+o_n(1)\|u_n\|
	=
	2I(u_n)-I'(u_n)[u_n]
	=\frac{2^*-2}{2^*-1}\|\phi_{u_n}\|^2
	\ge
	C_1|u_n|_2^{2^*}-C_2\|u_n\|^2
	\]
	so that
	\begin{equation*}\label{PS2}
	\|\phi_{u_n}\|^2\le C(1+\|u_n\|)
	\end{equation*}
	and
	\begin{equation*}
	|u_n|_{2}^{2^*}\le C(1+\|u_n\|^2).
	\end{equation*}
	Hence
	\begin{equation*}%\label{PS3}
	\|u_n\|^2
	= 2 I(u_n)+\lambda |u_n|_2^2+\frac{1}{2^*-1}\|\phi_{u_n}\|^2
	\le
	C\left(1+\|u_n\|^{\frac{4}{2^*}}+\|u_n\|\right)
	\end{equation*}
	and, since $2<2^*$, it follows that $(u_n)$ is bounded.\\ Then we can assume that there exists $u\in H^1_0(\Omega)$ such that $u_n \rightharpoonup u$ in $H^1_0(\Omega)$, $u_n \to u$ in $L^q(\Omega)$ for every $q\in [1, 2^*)$ and a.e. in $\Omega$. \\
	Now, let us set $f(s):=|s|^{2^*-3 }s$. Since $(u_n)$ is bounded in $L^{2^*}(\Omega)$, then $(f(u_n))$ is bounded $L^{\frac{2^*}{2^*-2}}(\Omega)$ and so, in a standard way, it follows that $f(u_n) \rightharpoonup f(u)$ in $L^{\frac{2^*}{2^*-2}}(\Omega)$.
	%{\color{red}Then, let $\varphi\in C^\infty_0(\Omega)$
	%\begin{equation}\label{PS4}
	%I'(u_n)[\varphi]=\int_\Omega \nabla u_n\nabla\varphi\, dx -\lambda\int_\Omega u_n\varphi\, dx -\int_\Omega f(u_n)\phi_{u_n}\varphi\, dx.
	%\end{equation}
	%By definition $$\int_\Omega\nabla u_n \nabla\varphi\, dx -\lambda\int_\Omega u_n\varphi\, dx \underbrace{\longrightarrow}_{n\to+\infty} \int_\Omega\nabla u \nabla\varphi\, dx -\lambda\int_\Omega u\varphi\, dx.$$ We claim that
	%$$\int_\Omega f(u_n)\phi_{u_n}\varphi\, dx\underbrace{\longrightarrow}_{n\to+\infty}\int_\Omega f(u)\phi_{u}\varphi\, dx.$$ Indeed,}
	Then, for all $\varphi\in C^\infty_0(\Omega)$, using also \eqref{vii} of Lemma \ref{Lemmaphi}, H\"older and Sobolev inequalities, and since $\phi_u\varphi\in L^{\frac{2^*}{2}}(\Omega)$,
	\begin{align*}
	\left|\int_\Omega f(u_n)\phi_{u_n}\varphi\, dx
	-\int_\Omega f(u)\phi_{u}\varphi\, dx\right|
	&\le \left|\int_\Omega \left(\phi_{u_n}-\phi_u\right)f(u_n)\varphi\, dx \right|
	+\left|\int_\Omega\left(f(u_n)-f(u)\right)\phi_u \varphi\, dx\right|\\
	&\le
	C |\varphi|_\infty\|u_n\|^{{2^*-2}}|\phi_{u_n}-\phi_u|_{\frac{2^*}{2}}
	+o_n(1)\longrightarrow 0
	\end{align*}
%	\begin{equation}\label{PS5}
%	\begin{aligned}
%	\end{aligned}
%	\end{equation*}
	as $n\to+\infty$.
	Hence
	$$I'(u_n)[\varphi] \to I'(u)[\varphi]$$
	and, by density, we get
	$$0=I'(u)[u]=\|u\|^2-\lambda|u|_2^2-\|\phi_u\|^2,$$
	from which
	\begin{equation}\label{posI}
	I(u)=\frac{2^*-2}{2(2^*-1)}\|\phi_u\|^2\ge 0.
	\end{equation}
	Since $u_n \rightharpoonup u$ in $H^1_0(\Omega)$ we get
	$$\|u_n\|^2 =\|u_n-u\|^2+\|u\|^2+o_n(1).$$
	Then, by using the strong convergence $u_n \to u$ in $L^2(\Omega)$  and \eqref{starstarstarp} we get
	\begin{equation}\label{J*}
	I(u_n)=I(u) + I_0(u_n-u)+o_n(1)
	\end{equation}
	with $$I_0(u)=\frac 12 \|u\|^2-\frac{1}{2(2^*-1)}\|\phi_u\|^2.$$ Furthermore
	\begin{equation}\label{J'*}
	\begin{aligned}
	o_n(1)&= I'(u_n)[u_n-u]=( I'(u_n)-I'(u))[u_n-u]\\
	&=\|u_n-u\|^2-\lambda|u_n-u|_2^2-\int_\Omega \phi_{u_n}f(u_n)(u_n-u)\, dx+\int_\Omega \phi_u f(u)(u_n-u)\, dx.
	\end{aligned}
	\end{equation}
	Since $u_n \rightharpoonup u$ in $L^{2^*}(\Omega)$ and $\phi_u f(u)\in L^{\frac{2^*}{2^*-1}}(\Omega)$,
	\[
	\int_\Omega \phi_u f(u)(u_n-u)\, dx =o_n(1).
	\]
	Moreover
	\begin{equation}
	\label{IV}
	\begin{aligned}
	\int_\Omega \phi_{u_n}f(u_n)(u_n-u)\, dx
	&=
	\int_\Omega \phi_{u_n} |u_n|^{2^*-1}\, dx
	-\int_\Omega \phi_u |u|^{2^*-1}\, dx
	\\
	&\quad-\int_\Omega \left(\phi_{u_n}-\phi_u\right) f(u_n) u\, dx
	-\int_\Omega \phi_u u \left(f(u_n)-f(u)\right)\, dx.
	\end{aligned}
	\end{equation}
	Since the sequence $((\phi_{u_n}-\phi_u) f(u_n))$ is bounded in $L^\frac{2^*}{2^*-1}(\Omega)$, $\phi_{u_n}\to\phi_u$ and $f(u_n)\to f(u)$ a.e. in $\Omega$, by \cite[Proposition 5.4.7]{W2013} we have
	\begin{equation}
	\label{V}
	\int_\Omega \left(\phi_{u_n}-\phi_u\right) f(u_n) u\, dx=o_n(1)
	\end{equation}
	and, analogously, we can prove that
	\begin{equation}
	\label{VI}
	\int_\Omega \phi_u u \left(f(u_n)-f(u)\right)\, dx=o_n(1).
	\end{equation}
	Then, using \eqref{starstarstarp}, \eqref{V} and \eqref{VI} in \eqref{IV}, we obtain
	\[
	\int_\Omega \phi_{u_n}f(u_n)(u_n-u)\, dx
	=\int_\Omega \phi_{u_n-u}|u_n-u|^{2^*-1}\, dx +o_n(1).
	\]
	Moreover, by \eqref{J'*} we get
	\begin{equation}
	\label{penultima}
	\|u_n-u\|^2 -\int_\Omega \phi_{u_n-u}|u_n-u|^{2^*-1}\, dx=o_n(1)
	\end{equation}
	and so
	$$I_0(u_n-u)=\frac 12 \|u_n-u\|^2 -\frac{1}{2(2^*-1)}\|u_n-u\|^2+o_n(1)= \frac{2}{N+2}\|u_n-u\|^2+o_n(1).$$ On the other hand, from \eqref{J*} we get, using also \eqref{posI},
	$$I_0(u_n-u)= I(u_n)-I(u)+o_n(1)\le  c+ o_n(1)<\frac{2}{N+2}S^{\frac N 2}.$$ Then it follows that 
	$$\limsup_n\|u_n-u\|^2 <S^{\frac N 2}$$
	and so, by \eqref{penultima} and \eqref{phiuu2*-1} of Lemma \ref{Lemmaphi},
	%since $\left(\int_\Omega \phi_u |u|^{2^*-1}\, dx \right)^{\frac 12}=\|\phi_u\|\le \frac{1}{S^{\frac{2^*}{2}}}\|u\|^{2^*-1}$ then
	$$
	o_n(1) =\|u_n-u\|^2-\int_\Omega \phi_{u_n-u}|u_n-u|^{2^*-1}\, dx
	\ge \|u_n-u\|^2\left[1-\left(\frac{\|u_n-u\|^{2}}{S^{\frac{N}{2}}}\right)^{2^*-2}\right]\ge C\|u_n-u\|^2.
	$$
	Hence $u_n \to u$ in $H^1_0(\Omega)$.
\end{proof}

\section{The dimension $N=3$: positive solutions}\label{N3}
In dimension $3$ this type of problems turns out to be rather delicate even in the Brezis-Nirenberg problem \eqref{eq:BN} (see \cite{BN}) and so we consider only the case in which $\Omega$ is a ball. Moreover, for simplicity but without losing generality, we take $\Omega=B_1$ so that $\lambda_1=\pi^2$.

Before we present our existence and nonexistence results, we introduce the following important symmetry property of all positive solutions, essentially based on the moving plane technique due to Gidas, Ni and Nirenberg \cite{GNN}. To prove it we use a standard argument suitably adapted for our purposes. We point out that the dimension $N$ does not have any role in the proof, even if we are here interested only in the case $N=3$.

    \begin{lemma}\label{le:GNN}
        If $\Omega=B_1$, any positive solution $(u,\phi)$ of \eqref{pb}, with $u, \phi \in C^2(B_1)\cap C(\bar {B_1})$, is such that both $u$ and $\phi$ are radially symmetric.
    \end{lemma}

    \begin{proof}
        Assume $u\in C^1(\bar{B_1})$ in order to compute the derivatives in $\partial B_1$ (this further assumption can be removed generalizing the meaning of the derivatives on the boundary) and $\lambda >0$ since we have no solution for $\lambda\le 0$.\\
Define for any $\mu \in [0,1]$ the sets
$$P_\mu=\{x\in\R^3\mid x_3=\mu\}$$
and
$$E_\mu=\{x\in B_1\mid \mu<x_3<1\}.$$
Moreover, for any $x=(x_1, x_2, x_3)\in\R^3$, let $x_\mu = (x_1, x_2, 2\mu - x_3)$ be the reflection of $x$ with respect to $P_{\mu}$ and assume the following notation for any $x\in E_\mu$:
$$u_{\mu}(x)=u(x_\mu).$$
By the Hopf Lemma (see for example \cite[Lemma 3.4]{GT}), we have that for all $\mu\in ]0,1[$ the outward derivative $\frac {\partial u}{\partial \nu}$ is negative in $\partial E_\mu\cap \partial B_1$. Of course, since $\nabla u$ and $\nu$ have the same direction, this means that the verse of $\nabla u$ is inward, and then we deduce that $\frac{\partial u}{\partial x_3}<0$ in $\partial E_\mu\cap \partial B_1$.\\
By continuity, we deduce that for values of $\mu$ close to 1, we have $\frac{\partial u}{\partial x_3}<0$ everywhere in $E_\mu\cup \tilde{E_\mu}$, where $\tilde{E_\mu}=\{x\in B_1\mid x_\mu\in E_\mu\}$. Of course, this fact implies that values of $\mu$ close to 1 are in $\mathcal A:= \{\mu\in ]0,1] \mid \forall x\in E_\mu: u(x)<u_\mu(x) \}$.\\
Now, define $\mu_0=\inf\{\mu\in]0,1]\mid \eta\in\mathcal A,\,\forall \eta\in[\mu,1]\}$.\\
We want to prove by contradiction that $\mu_0=0$ so that we assume $\mu_0>0$.\\
Observe that in $E_{\mu}$ we have $-\Delta u_{\mu}=(-\Delta u)_{\mu}$, $\frac{\partial u_{\mu}}{\partial x_3}=-\left(\frac{\partial u}{\partial x_3}\right)_\mu$ and  $\phi_{u_{\mu}} = (\phi_u)_{\mu}$.\\
Denote $\phi=\phi_u$, $\phi_\mu=(\phi_u)_\mu$ and $w=u_{\mu_0} - u$.\\
Of course we have  $w\ge 0$ in $\overline{E_{\mu_0}}$   and
    \begin{equation}\label{eq:w}
        -\Delta w = \lambda w+ u_{\mu_0}^{2^*-2}\phi_{\mu_0} - u^{2^*-2}\phi
        \hbox{ in } E_{\mu_0}.
    \end{equation}
Since we have
    \begin{align*}
        &-\Delta (\phi_{\mu_0}-\phi)= u_{\mu_0}^{2^*-1} - u^{2^*-1}\ge 0\hbox{ in } E_{\mu_0},\\
        &\phi_{\mu_0}-\phi\ge 0\hbox{ on } \partial E_{\mu_0},
    \end{align*}
by the weak maximum principle we deduce that $\phi_{\mu_0}\ge \phi$ in $E_{\mu_0}$.\\
Now, by \eqref{eq:w}, we have that $-\Delta w \ge 0$ in $E_{\mu_0}$ and then, by the strong maximum principle and since $w=0$ in $P_{\mu_0}$, we have that either $w>0$ in $E_{\mu_0}$ or $w$ is constant. Of course $w$ is not constant since $w>0$ on $\partial E_{\mu_0}\setminus P_{\mu_0}$ and then $w$ is positive in $E_{\mu_0}$.\\
Moreover, by the Hopf Lemma, in $P_{\mu_0}$ we have $\frac{\partial w}{\partial x_3}>0$  and then, since $x=x_{\mu_0}$, we get $\frac{\partial u}{\partial x_3}=-\frac{1}2\frac{\partial w}{\partial x_3}<0$.\\
Then, again by continuity, there exists $V$ a neighborhood of $P_{\mu_0}$ such that $\frac{\partial u}{\partial x_3}<0$ in $V\cap B_1$. \\
In conclusion, since $u<u_{\mu_0}$ in $E_{\mu_0}$ and $\frac{\partial u}{\partial x_3}<0$ in $V\cap B_1$, we deduce that actually there exists $\e>0$ such that $\mu_0-\e>0$ and $[\mu_0-\e,1]\subset \mathcal A$: a contradiction.\\
We deduce that $\mu_0=0$ which implies
	\begin{equation}\label{eq:mov1}
		u\le u_0\hbox{ in }E_0.
 	\end{equation}
Since the same arguments can be repeated replacing the interval $]0,1]$ with $[-1,0[$, $E_\mu$ with $F_\mu=\{x\in B_1\mid -1< x_3 < \mu\}$ and defining $\mu_0=\sup\{\mu\in[-1,0[\,\mid \eta\in\mathcal A,\,\forall \eta\in[-1,\mu]\}$, we can prove that
	\begin{equation}\label{eq:mov2}
		u\le u_0\hbox{ in }F_0.
 	\end{equation}
It is easy to see that \eqref{eq:mov1} and \eqref{eq:mov2} together imply that $u=u_0$ in $B_1$. Since this conclusion can be obtained in the same way for any reflection plane containing the origin, we deduce that $u$ is radial.
The symmetry of $\phi$ comes easily from that of $u$.
    \end{proof}

\subsection{An existence result}\label{sb31}
This section is devoted to the proof of (\ref{i12}) of Theorem \ref{Th1}.
%The first part follows the same steps as in \cite{AzzDav}, by looking for critical points of $I$.
Since the geometrical assumptions of the Mountain Pass Theorem are satisfied, we set
$$c:=\inf_{\gamma\in \Gamma}\max_{t\in [0, 1]} I(\gamma(t)),$$ where $\Gamma=\left\{\gamma\in C([0, 1], H^1_0(\Omega)): \gamma(0)=0, I(\gamma(1))<0\right\}$ and, by Lemma \ref{PS}, we only have to prove that
$$c< \frac 2 5 S^{\frac 3 2}=\frac{3\sqrt{3}}{10}\pi^2.$$
%{\color{red}, and, to show that there exists a nontrivial solution to the problem \eqref{pb}, we consider a Palais-Smale (PS) sequence $(u_n)$ at the mountain pass level $c$. As done in \cite[Step 1 of the Proof of Theorem 1.1]{AzzDav}, we can show that $(u_n)$ is bounded and, if by contradiction we assume $u_n\rightharpoonup 0$, then
%\begin{equation}\label{levelc}
%c\ge \frac {2}{5} S^{\frac 32}.
%\end{equation}}
So we consider a smooth positive function $\varphi=\varphi(r)$ such that $\varphi(0)=1$, $\varphi'(0)=0$ and $\varphi(1)=0$.
Following \cite{BN} we set $r=|x|$,
$$u_\e(r)=\frac{\varphi(r)}{(\e+r^2)^{\frac 12}}$$
and we get the estimates
$$\| u_\e\|^2= K_1\e^{-\frac 12}
+4\pi\int_0^1 |\varphi'(r)|^2\, dr
+O(\e^{\frac 12}),
\qquad
|u_\e|_2^2 =4\pi\int_0^1 \varphi^2(r)\, dr+O(\e^{\frac 12}),$$ where
$$K_1:=\int_{\mathbb R^3}\frac{|x|^2}{(1+|x|^2)^3}\, dx
%{\color{red}=\omega \int_0^{+\infty}\frac{r^4}{(1+r^2)^3}\, dr }
= \frac{3}{4}\pi^2.$$
%{\color{red}and $\omega$ is the area of the unitary sphere of dimension 2.}
%A simple computation shows that
%$$
%S
%=\frac{K_1}{\left(\displaystyle\int_{\mathbb R^3} \frac{1}{(1+|x|^2)^3}\, dx\right)^{\frac 1 3}}
%{\color{red}= \frac{\frac{3}{16}\pi\omega}{\left(\displaystyle\omega\int_0^{+\infty}\frac{r^2}{(1+r^2)^3}\, dr\right)^{\frac 13}}
%	=\frac{\frac{3}{16}\pi\omega}{\left(\frac{1}{16}\pi\omega\right)^{\frac 13}}}
%=3 2^{-\frac 8 3}\pi^{\frac 2 3}\omega^{\frac 2 3}
%$$
%{\color{red}from which it follows that $$S^{\frac 32}=\frac{3\sqrt 3}{16}\pi\omega.$$}
%In order to get the contradiction, taking into account that $S=\frac 3 2 \pi\sqrt[3]{\frac{\pi}{2}}$, we are going to give an estimate of the value $c$, and in particular we show that
We have that
\begin{equation}
\label{supte}
\sup_{t>0} I(t u_\e)= I(t_\e u_\e)=\frac 2 5 \frac{\left(\|u_\e\|^2-\lambda  |u_\e|_2^2\right)^{\frac 54} }{\| \phi_{u_\e}\|^{\frac 12}}
\end{equation}
where
%$t_\e$ is the unique positive solution of the equation $$\frac{d}{dt} I(t u_\e)=0.$$
%Since $$\frac{d}{dt} I(t u_\e)= t \|u_\e\|^2-\lambda t |u_\e|_2^2- t^9 \| \phi_{u_\e}\|^2,$$ we have that
$$t_\e:=\sqrt[8]{\frac{\|u_\e\|^2-\lambda  |u_\e|_2^2 }{\| \phi_{u_\e}\|^2}}.$$
Moreover
\begin{equation}\label{num}
\left(\|u_\e\|^2-\lambda  |u_\e|_2^2\right)^{\frac 54}
%=\left(\e^{-\frac 12} K_1 +A(\varphi)+O(\e^{\frac 12})\right)^{\frac 54}
= K_1^\frac{5}{4}\e^{-\frac 5 8}\left(1+ \frac 54\frac{A(\varphi)}{K_1}\sqrt{\e} + o(\sqrt{\e})\right)
\end{equation}
where
$$A(\varphi):=4\pi\left(\int_0^1 |\varphi'(r)|^2\, dr -\lambda \int_0^1 \varphi^2(r)\, dr\right).$$
Now, since $u_\e$ is radial and smooth, also $\phi_{u_\e}$ is radial and smooth. Moreover, for $s\in(0,1)$
$$\phi_{u_\e}''(s)+2\frac{\phi_{u_\e}'(s)}s=-u^5_\e(s),$$
namely,
$$(s^{2}\phi_{u_\e}'(s))'=-s^2u^5_\e(s).$$
%{\color{red}Integrating in $(0,t)$, for $t\in(0,R)$, and using the smoothness of $\phi_{u_\e}$ ($\phi_{u_\e}'(0)=0$), we get
%$$\phi_{u_\e}'(t)=-\frac 1{t^2}\int_0^t s^{2}u^5_\e(s)\,ds.$$
%Then, integrating in $(r,R)$, for $r\in(0,R)$, and since $\phi(R)=0$, we have
%\begin{align*}
%\phi_{u_\e}(r)
%&=\int_r^R\frac 1{t^2}\int_0^t s^2u^5_\e(s)\,ds\,dt\\
%&=\int_0^r\left(\int_r^R\frac{s^2u^5_\e(s)}{t^2}\,dt\right)\,ds+\int_r^R\left(\int_s^R\frac{s^2u^5_\e(s)}{t^2}\,dt\right)\,ds\\
%&=\int_0^rs^2u^5_\e(s)\left(\frac 1 r -\frac 1 R\right)\,ds+\int_r^Rs^2u^5_\e(s)\left(\frac 1 s -\frac 1 R \right)\,ds\\
%&=\frac 1 r\int_0^R s \min\{r,s\}u^5_\e(s)\, ds-\frac 1 R\int_0^Rs^2u^5_\e(s)\,ds.
%\end{align*}}
%	Stima diretta di $\phi_u'$\\
%	Since
%	\begin{align*}
%		\phi_{u_\e}(r)
%		&=
%		\frac{1}{r}\int_0^R s\min\{r, s\} |u_\e(s)|^5\, ds-\frac 1 R\int_0^Rs^2|u_\e(s)|^5\,ds\\
%		&=
%		\frac{1}{r}\int_0^r s^2 |u_\e(s)|^5\, ds
%		+\int_r^{R} s |u_\e(s)|^5\, ds
%		-\frac 1 R\int_0^Rs^2|u_\e(s)|^5\,ds
%	\end{align*}
%	we have that
Integrating in $(0,r)$, for $r\in(0,1)$, and using the smoothness of $\phi_{u_\e}$ ($\phi_{u_\e}'(0)=0$), we get
	\begin{align*}
		\phi'_{u_\e}(r)
		&=
		-\frac{1}{r^2} \int_0^r s^2 |u_\e(s)|^5\, ds \\
		&=
		-\frac{1}{r^2} \int_0^r \frac{\varphi^5(s)s^2}{(\e+s^2)^\frac{5}{2}}\, ds\\
		&=
		-\frac{1}{r^2} \left(\int_0^r  s^2\frac{\varphi^5(s)-1}{(\e+s^2)^\frac{5}{2}}\, ds
		+ \int_0^r \frac{s^2}{(\e+s^2)^\frac{5}{2}}\, ds\right)
	\end{align*}
	and so, using the fact that $\varphi(0)=1$ and $\varphi'(0)=0$,
	\begin{align*}
		\|\phi_{u_\e}\|^2
		&=
		4\pi \int_0^1 \frac{1}{r^2} \left(\int_0^r  s^2\frac{\varphi^5(s)-1}{(\e+s^2)^\frac{5}{2}}\, ds
		+ \int_0^r \frac{s^2}{(\e+s^2)^\frac{5}{2}}\, ds\right)^2\,dr\\
		&\ge
		4\pi
		\left[
		\int_0^1 \frac{1}{r^2} \left(
		\int_0^r \frac{s^2}{(\e+s^2)^\frac{5}{2}}\, ds\right)^2\,dr\right.\\
		&\qquad\left.+ 2 \int_0^1 \frac{1}{r^2} \left(\int_0^r  s^2\frac{\varphi^5(s)-1}{(\e+s^2)^\frac{5}{2}}\, ds\right)
		\left(\int_0^r \frac{s^2}{(\e+s^2)^\frac{5}{2}}\, ds\right)\,dr
		\right]
		\\
		&\ge
		4\pi
		\left[
		\int_0^1 \frac{1}{r^2} \left(
		\int_0^r \frac{s^2}{(\e+s^2)^\frac{5}{2}}\, ds\right)^2\,dr\right.\\
		&\qquad\left.-C \int_0^1 \frac{1}{r^2} \left(\int_0^r  \frac{s^4}{(\e+s^2)^\frac{5}{2}}\, ds\right)
		\left(\int_0^r \frac{s^2}{(\e+s^2)^\frac{5}{2}}\, ds\right)\,dr
		\right]
		\\
		&=
		4\pi
		\left[
		\frac{1}{9\e^2} \int_0^1 \frac{r^4}{(\e+r^2)^3} \,dr\right.\\
		&\qquad\left.-\frac{C}{3\e} \int_0^1  \left(\log\left(\frac{r+\sqrt{\e+r^2}}{\sqrt{\e}}\right)
		- \frac{r(4r^2+3\e)}{3(\e+r^2)^\frac{3}{2}}\right)
		\frac{r}{(\e+r^2)^\frac{3}{2}}\,dr
		\right]\\
%		&=
%		\omega
%		\left[
%		\frac{1}{9\e^2} \int_0^R \frac{r^4}{(\e+r^2)^3} \,dr
%		-\frac{C}{3\e} \int_0^R  \frac{r}{(\e+r^2)^\frac{3}{2}} \log\left(\frac{r+\sqrt{\e+r^2}}{\sqrt{\e}}\right) \,dr
%		+\frac{C}{9\e} \int_0^R \frac{r^2(4r^2+3\e)}{(\e+r^2)^3}\,dr
%		\right]\\
%		&=
%		\omega
%		\left[
%		\frac{1}{72\e^\frac{5}{2}}
%		\left(\frac{3}{2}\pi -3\arctan \frac{\sqrt{\e}}{R} - \frac{R(5R^2+3\e)}{(\e+R^2)^2}\e\right)\right.\\
%		&\qquad\left.
%		-\frac{C}{3\e^\frac{3}{2}}
%		\left(\frac{\pi}{2}-\arctan\frac{\sqrt{\e}}{R} -\frac{\sqrt{\e}}{\sqrt{\e+R^2}}\log\left(\frac{R+\sqrt{\e+R^2}}{\sqrt{\e}}\right)\right)\right.\\
%		&\qquad\left.
%		+\frac{C}{72\e^\frac{3}{2}} \left(\frac{15}{2}\pi-15\arctan\frac{\sqrt{\e}}{R}-\frac{R(17R^2+15\e)}{(\e+R^2)^2}\sqrt{\e}\right)
%		\right]\\
		&=
		\frac{\pi^2}{12}\e^{-\frac{5}{2}}\left(1 - \frac{16}{3\pi}\sqrt{\e} + o(\sqrt{\e})\right)
	\end{align*}
		Hence
		\begin{equation}
		\label{den}
		\| \phi_{u_\e}\|^\frac{1}{2}
		\ge
		\sqrt[4]{\frac{\pi^2}{12}}\e^{-\frac 5 8}
		\left(1-\frac{16}{3\pi}\sqrt{\e}+ o(\sqrt{\e})\right)^{\frac 14 }
		=
		\sqrt[4]{\frac{\pi^2}{12}}\e^{-\frac 5 8}
		\left(1-\frac{4}{3\pi}\sqrt{\e}+ o(\sqrt{\e})\right).
		\end{equation}
		Then, using \eqref{num} and \eqref{den} in \eqref{supte}, we get
		$$%\begin{aligned}
		\sup_{t>0} I(t u_\e)
		%&= I(t_\e u_\e)=\frac 2 5 \frac{\left(\|u_\e\|^2-\lambda  |u_\e|_2^2\right)^{\frac 54} }{\left(|\nabla \phi_{u_\e}|_2^2\right)^{\frac 14}}
		\le
		%\frac 25 \frac{K_1^{\frac 54}+ \textcolor{red}{\frac 54} K_1^{\frac 1 4} A(\varphi)\e^{\frac 12}+ o(\e^{\frac 12})}{\left(\frac{\pi}{48}\omega\right)^{\frac 14}\textcolor{red}{-\frac{2\omega^{\frac 14}}{3^{\frac 54}\pi^{\frac 34}R}}\e^{\frac 12}+o(\e^{\frac 12})}
		%\\
		%&
		%=
		\frac 2 5 S^{\frac 32}\left(1+\frac{5A(\varphi)+4\pi}{3\pi^2}\sqrt{\e}+o(\sqrt{\e})\right).
		%&=\frac 25 \frac{K_1^{\frac 54}}{\left(\frac{\pi}{48}\omega\right)^{\frac 14}} +\e^{\frac 12}\left(\textcolor{red}{\frac 54 %K_1^{\frac 1 4}A(\varphi)+\frac{1}{4}\(\frac{\pi\omega}{48}\right)^{-\frac 34}\frac{\omega}{9}}\right)+ O(\e)\\
		%&=\frac 2 5\frac{3\sqrt 3}{16}\pi\omega+\e^{\frac 12}\left(\textcolor{red}{\frac 54 K_1^{\frac 1 4}\(\frac{\pi\omega}{48}\)^{\frac 14}A(\varphi)+\frac{1}{4}\(\frac{\pi\omega}{48}\right)^{-\frac 34}\frac{\omega}{9}K_1^{\frac 54}}\right)+ O(\e)\\
		%&=\frac 2 5 S^{\frac 32}+\e^{\frac 12}\left(\textcolor{red}{\frac 54 K_1^{\frac 1 4}A(\varphi)+\frac{1}{4}\(\frac{\pi\omega}{48}\right)^{-\frac 34}\frac{\omega}{9}}\right)+ O(\e).
		%\end{aligned}
		$$
%		where
%		\[
%		B(\varphi)=\frac{\sqrt{3}}{4}\left(\frac{3\pi}{2R}+5A(\varphi)\right).
%		\] %is a number which has the same sign of $5A(\varphi)+\frac{16K_1}{3\pi R}=5A(\varphi)+\frac\omega R$.
		If we take $\varphi(r)=\cos\(\frac{\pi r}{2}\)$, then $A(\varphi)=2\pi  \left(\frac{\pi^2}{4}-\lambda\right)$ and we conclude assuming that
		\[
		\lambda>\lambda_1\left(\frac{1}{4}+\frac{2}{5\pi^2}\right).
		%\left(\frac{\pi^2}4+\frac 25\right)\frac 1{R^2}=\frac{\lambda_1}{4}+\frac 2{5R^2}=\lambda_1\left(\frac 14 + \frac 2 {5\pi^2}\right)$.
		\]
		
It can be showed that, applying usual arguments, we are allowed to assume $u$ nonnegative. As a trivial consequence of the strong maximum principle applied to both the equations, we actually deduce that the solution is positive and, by Lemma \ref{le:GNN}, radial.
			
Finally we prove it is a ground state arguing as in  of \cite[Step 2 of the Proof of Theorem 1.1]{AzzDav}.

\subsection{A nonexistence result}
In this section we prove the (\ref{ii12}) of Theorem \ref{Th1}.
	We reason as in \cite{BN}, taking into account that, by Lemma \ref{le:GNN}, we are allowed to consider the ODE radial formulation of our problem.
	 So, we write $u(x)=u(r)$ and $\phi(x)=\phi(r)$ where $r=|x|$ and we assume by contradiction $(u, \phi)$ is a positive solution of
	\begin{align}
	&-u''-\frac 2 r u' =\lambda u + \phi u^4\qquad \mbox{on}\,\, (0, 1)\label{eq1}\\
	&-\phi''-\frac 2 r \phi' = u^5\qquad \mbox{on}\,\, (0, 1)\label{eq2}
	\end{align}
with the following boundary conditions
\[
u'(0)=u(1)=0,%\label{eq1cond}\\
\qquad
\phi'(0)=\phi(1)=0.%\label{eq2cond}
\]
	Let $\psi$ be any smooth function such that $\psi(0)=0$.\\
	We multiply \eqref{eq1} by $r^2\psi u'$, we integrate on $(0, 1)$ and we get
%	$$\int_0^1 -r^2 \psi \left[\frac{(u')^2}{2}\right]'\, dr-\int_0^1 2 r \psi(u')^2\, dr =\int_0^1\lambda r^2 \psi \left[\frac{(u)^2}{2}\right]'\, dr+\int_0^1 r^2\psi\phi\left[\frac{(u)^5}{5}\right]'\, dr.$$
%	We integrate by parts and we get
\begin{equation}\label{e:1}
\begin{split}
-\frac{\psi(1)(u'(1))^2}{2}
& +\int_0^1\left(\frac{r^2}{2}\psi'-r\psi\right)(u')^2\, dr
=-\frac{\lambda}{2}\int_0^1 (r^2\psi'+2r\psi)u^2\, dr  \\
&\quad
-\frac 15 \int_0^1 (r^2\psi'+2r\psi)\phi u^5\, dr-\frac 15 \int_0^1 r^2 u^5 \psi\phi'\, dr.
\end{split}
\end{equation}
	We multiply \eqref{eq2} by $r^2\psi \phi'$, we integrate by parts on $(0, 1)$ and  we get
	\begin{equation}\label{e:2}
	-\frac{\psi(1)(\phi'(1))^2}{2}+\int_0^1\left(\frac{r^2}{2}\psi'-r\psi\right)(\phi')^2\, dr = \int_0^1 r^2 u^5 \psi\phi'\, dr.
	\end{equation}
	Then, by \eqref{e:1} and \eqref{e:2} we obtain
	\begin{equation}\label{e:3}
	\begin{split}
	 -\frac{\psi(1)(u'(1))^2}{2}&-\frac{\psi(1)(\phi'(1))^2}{10}\\
	 &\quad+\int_0^1\left(\frac{r^2}{2}\psi'-r\psi\right)(u')^2\, dr +\frac 15\int_0^1\left(\frac{r^2}{2}\psi'-r\psi\right)(\phi')^2\, dr \\
	 &\quad=-\frac{\lambda}{2}\int_0^1 (r^2\psi'+2r\psi)u^2\, dr-\frac 15 \int_0^1 (r^2\psi'+2r\psi)\phi u^5\, dr.
	\end{split}
	\end{equation}
	Now, multiplying \eqref{eq1} by $\left(\frac{r^2}{2} \psi' -r\psi\right) u$ and integrating in $(0, 1)$ we have
	\begin{equation}\label{f:1}
	-\frac 14 \int_0^1 r^2 \psi''' u^2\, dr +\int_0^1 \(\frac{r^2}{2}\psi'-r\psi\)(u')^2\, dr
	=  \int_0^1    \(\frac{r^2}{2} \psi'- r\psi\)(\lambda u^2 + \phi u^5)\, dr
	\end{equation}
	and, multiplying \eqref{eq2} by $\(\frac{r^2}{2}\psi' -r\psi\right)\phi $ and integrating in $(0, 1)$, we have
	\begin{equation}\label{f:2fin}
	\begin{aligned}
	\int_0^1  \(\frac{r^2}{2}\psi' -r\psi\right)\(\phi'\)^2\, dr
	-\frac 1 4 \int_0^1 \phi^2 r^2 \psi'''\, dr
	=\int_0^1  \(\frac{r^2}{2}\psi'  -r\psi\right)\phi u^5\, dr.
	\end{aligned}
	\end{equation}
	Combining \eqref{e:3}, \eqref{f:1} and \eqref{f:2fin} we get
	\begin{equation}\label{fin}
	\begin{split}
	\int_0^1  r^2 \(\frac 14 \psi''' +\lambda \psi'\)u^2\, dr
	& +\frac{1}{20}\int_0^1 r^2 \phi^2  \psi'''\, dr \\
	& =\frac{\psi(1)(u'(1))^2}{2} +\frac{\psi(1)(\phi'(1))^2}{10} +\frac 45\int_0^1 u^5\phi r(\psi-r\psi')\, dr.
	\end{split}
	\end{equation}

Our aim is to find those $\lambda>0$ for which there exists a smooth function $\psi_\lambda$ such that $\psi_\lambda(0)=0$ and \eqref{fin} does not hold.

As a first step, consider $0<\lambda \le \frac{\lambda_1}{16}$ and set $\psi_\lambda(r)= \sin (\sqrt{4\lambda}r)$ as in \cite{BN}. We have $\psi_\lambda(1)\ge 0$ and $\frac14 \psi_\lambda''' +\lambda \psi_\lambda'=0$.
	Moreover, for any $r\in]0, 1]$,
	$$
	\psi_\lambda'''(r) \le 0 < \psi_\lambda(r)-r\psi_\lambda'(r)
	$$
and then equality \eqref{fin} is violated.

Now, for $k>0$ and $\lambda\in J:=\big]\frac{\lambda_1}{16},\frac{\lambda_1}{4}\big[$,  consider
\[
\psi_{k,\lambda}(r)= 1 - 2\lambda r^2 + k \sin (\sqrt{4\lambda}r) - \cos(\sqrt{4\lambda}r)
\]
and observe that, for all $r\in]0, 1]$,
    \begin{equation*}%\label{eq:contr}
        \frac14 \psi_{k,\lambda}'''(r) +\lambda \psi_{k,\lambda}'(r)< 0 <\psi_{k,\lambda}(r)-r\psi_{k,\lambda}'(r).
    \end{equation*}
Moreover, for $\lambda\in J$, we have
	\[
	\sup_{r\in]0,1]}\psi'''_{k,\lambda}(r)\le 0 \iff
	\psi'''_{k,\lambda}(1)\le 0 \iff
	k\le -\tan(\sqrt{4\lambda})
	\]
	and
	\[
	\psi_{k,\lambda}(1)\ge 0 \iff
	k\ge \frac{\cos(\sqrt{4\lambda}) - 1 + 2\lambda}{\sin(\sqrt{4\lambda})}.
	\]
	Then
	\[
	\frac{\cos(\sqrt{4\lambda}) - 1 + 2\lambda}{\sin(\sqrt{4\lambda})} \le -\tan(\sqrt{4\lambda}),
	\]
	namely
	\[
	1 + (2 \lambda-1) \cos(\sqrt{4\lambda})\ge 0,
	\]
	which holds if $\lambda\in\big]\frac{\lambda_1}{16},\lambda^*\big]$, where $\lambda^*$ is the unique solution of $1 + (2 \lambda-1) \cos(\sqrt{4\lambda})= 0$ in $J$.

\section{The dimensions $N\ge 4$}\label{N4}
In this section we are interested in studying \eqref{pb} for $N\ge 4$ in the resonance and in the nonresonance case in order to provide the proofs of Theorem \ref{Th2} and Theorem \ref{Th3}.\\
In particular, first we use similar arguments as those in Section \ref{sb31} to prove the existence of a positive solution at the mountain pass level when $0<\lambda<\lambda_1$. Then, after checking the geometrical assumptions of the Linking Theorem in \cite{Rab}, we show that sign changing solutions exist for $\lambda\ge\lambda_1$ both in the resonance and, provided $N\ge 6$, in the nonresonance case.

\subsection{Positive solutions}
We are reduced to prove that $c<\frac{2}{N+2} S^{\frac{N}{2}}$.
As in \cite{BN} we define
\begin{align*}
u_\e(x)=\frac{\varphi(|x|)}{(\e+|x|^2)^{\frac {N-2}2}},
\end{align*}
where $\varphi$ is a smooth positive function compactly supported in $\Omega$ such that $\varphi(x)=1$ in some neighborhood of $0$.
We have
\[
|\nabla u_\e|^2_2 = \frac{K_1}{\e^{\frac {N-2}2}} +O(1),
\qquad
|u_\e|_{2^*}^2=\frac{K_2}{\e^{\frac{N-2}2}}+ O(\e)
\]
and
\[
|u_\e|_2^2 =\left\{ 		
\begin{array}{ll}
K_3\e^{-\frac{N-4}2}+O(1)& \hbox{if } N\ge 5\\
K_3|\log\e|+ O(1)& \hbox{if } N =4
\end{array}
\right.
\]
where $K_1,K_2,K_3$ are positive constants and $K_1/K_2=S$ (see \cite{BN}).\\
Then, using \eqref{J}, we have the following estimate of the mountain pass level
\[
c
\le \max_{t>0}I(tu_\e)
\le \frac{2}{N+2}\left(\frac{|\n u_\e|_2^2-\frac{N+2}{2N}\l |u_\e|_2^2}{|u_\e|_{2^*}^2}\right)^\frac{N}{2}
\le \frac{2}{N+2}\Big(S+\e^{\frac{N-2}2}(O(1)-\psi(\e))\Big)^{\frac{N}{2}}
\]
where $\psi$ is a function such that $\lim_{\e\to 0}\psi(\e)=+\infty$.\\
Of course, if $\e$ is sufficiently small, we conclude.

\subsection{The geometrical properties for the Linking Theorem}
In this section we investigate the geometrical properties of the functional $I$ in order to verify the assumptions of the Linking Theorem. 
In what follows, we recall the geometrical construction developed in \cite{GR} to find  sign changing solutions for $\lambda\ge\lambda_1$.

Let us denote by $e_j$ the eigenfunctions relative to $\lambda_j\in \sigma(-\Delta)$ such that $|e_j|_2=1$. \\
Let $k\in\mathbb N$ and let us define $$H^-:={\rm span}\left\{e_j: j=1, \ldots, k\right\}; \qquad H^+:= \left(H^-\right)^\bot.$$
Without loss of generality let $0\in\Omega$ and let us take $m$ so large such that $B_\frac{2}{m}\subset \Omega$. Consider the function $\zeta_m : \Omega \to \mathbb R$ defined by

\begin{equation*}
\zeta_m(x):=
\begin{cases}
0
& \mbox{if } x\in B_\frac{1}{m}\\
m|x|-1
&\mbox{if } x \in A_m:= B_\frac{2}{m}\setminus B_\frac{1}{m}\\
1
&\mbox{if } x\in \Omega\setminus B_\frac{2}{m}.
\end{cases}
\end{equation*}
Let $e_j^m :=\zeta_m e_j$ be the approximating eigenfunctions and let $$H^-_m :={\rm span}\{e_j^m\,\,:\,\, j=1, \ldots, k\}.$$
We have the following
\begin{lemma}\label{lemma2}
As $m\to\infty$ we have 
\begin{equation}
\label{comevuoi}
e_j^m \to e_j \mbox{ in } H^1_0(\Omega)
\quad
\hbox{ and }
\max_{w\in H^-_m;  |w|_2=1}\|w\|^2 \le \lambda_k +c_k m^{2-N}
\end{equation}
where $c_k>0$.\\
Moreover, if $\lambda\ge\lambda_k$,
\begin{equation}
\label{ahah}
\sup_{w\in H^-_m}I(w)=O(m^{-\frac{N^2-4}{4}}).
\end{equation}
\end{lemma}
\begin{proof}
The first part is proved in \cite[Lemma 2]{GR}.\\
Moreover for all $m\in\mathbb{N}^*$ and $w\in H^-_m$, by \eqref{comevuoi} and (\ref{viii}) of Lemma \ref{Lemmaphi},
\[
I(w)
\le \frac{c_k m^{2-N}}{2(\l_k+c_k m^{2-N})} \|w\|^2-C \|w\|^{2(2^*-1)}
%&\le  m^{2-N}O(\|w^m\|^2)- O(\|w^m\|^{2(2^*-1)})\\
 \le C m^{-\frac{N^2-4}{4}}.
\]
\end{proof}
We consider the family of functions
\begin{equation}\label{bubble}
u_\e^*(x):=\alpha_N \frac{\e^{\frac{N-2}{2}}}{\left(\e^2+|x|^2\right)^{\frac{N-2}{2}}}, \quad \e>0,
\end{equation}
where $\alpha_N:=[N(N-2)]^{\frac{N-2}{4}}$, which solve
$$
\begin{cases}
-\Delta u =u^{2^*-1} & \mbox{in }\mathbb R^N\\
u\in H^1(\mathbb R^N)
\end{cases}
$$
and satisfy $$\int_{\R^N}|\n u_\e^*|^2 = \int_{\R^N}|u_\e^*|^{2^*}=S^{\frac N 2}$$
for all $\e>0$.\\
Let $\eta\in C^\infty_c( B_\frac{1}{m})$ be a positive cut-off function such that $\eta\equiv 1 $ in $ B_\frac{1}{2m}$, $\eta\le 1$ in $B_\frac{1}{m}$, $|\nabla\eta|_\infty\le 4m$ and consider
\begin{equation}
\label{uem}
u_\e(x):=\eta(x) u_\e^*(x).
\end{equation}
As $\e\to 0$ we have the following estimates due to Brezis and Nirenberg \cite{BN}
\begin{equation}\label{stimebn}
\|u_\e\|^2 =S^{\frac N 2}+O\(\e^{N-2}\),\qquad |u_\e|_{2^*}^{2^*}= S^{\frac N 2}+O\(\e^N\).
\end{equation}
For $v\in H_m^- \oplus \mathbb R^+\{u_\e\}$ we write $v=w+t u_\e$. By definition
\begin{equation*}%\label{supp}
{\rm supp}(u_\e) \cap {\rm supp}(w) = \emptyset.
\end{equation*}

We have

     \begin{lemma}\label{le:almsplit}
        If $u,w\in \H(\Omega)$ are such that $\supp(u)\cap\supp(w)=\emptyset $,
        then for $v=u+w$ we have
\begin{equation}
\label{Isplit1}
I(v)= I(u) + I(w) - \frac{1}{2^*-1}\int_{\Omega} \phi_w |u|^{2^*-1}\,dx.
\end{equation}
        %    \begin{equation}
           %     I(v)\le I(w) + I(u) + C R(w,u)
            %\end{equation}
        %where $$R(w,u)=\displaystyle\int_{S_w\times S_u}H(x, y)|w(x)|^{2^*-1}|u(y)|^{2^*-1}\,dx\,dy+\displaystyle\int_{S_u\times S_w}H(x, y)|u(x)|^{2^*-1}|w(y)|^{2^*-1}\,dx\,dy.$$
    \end{lemma}

    \begin{proof}
        Of course,
        $\| v\|^2=\| w\|^2 + \| u\|^2$ and $|v|^2_2=|w|^2_2 + |u|^2_2$.\\
Moreover, if $u,w\in \H(\Omega)$ have disjoint supports, then
\[
-\Delta \phi_v
=
|u|^{2^*-1} + |w|^{2^*-1}
=
-\Delta(\phi_u + \phi_w)\quad\hbox{in } \Omega
\]
and $\phi_v= 0 =\phi_u+\phi_w$ on $\partial \Omega$.\\
By uniqueness,
\[
\phi_v=\phi_u+\phi_w.
\]
So we have
	\begin{align*}
		I(v)%&=\frac 1 2 \| v\|^2-\frac \l 2 |v|_2^2 -\frac 1{2(2^*-1)}\|\phi_v\|^2\\
		&=\frac 1 2 \| u\|^2 + \frac 1 2 \|w\|_2^2 - \frac \l 2 |u|_2^2 -\frac \l 2 |w|_2^2 -\frac 1{2(2^*-1)}\|\phi_u + \phi_w\|^2\\
		&=I(u)+I(w)-\frac 1{2^*-1}\int_{\Omega}\n \phi_u\n \phi_w\,dx\\
		&=
		I(u) + I(w) - \frac{1}{2^*-1}\int_{\Omega} \phi_w |u|^{2^*-1}\,dx
	\end{align*}
        %Moreover, we have
           % \begin{align*}
              %  \into \phi_v|v|^{2^*-1}\, dx &= \gamma_N\int_{\Omega\times \Omega} \frac{|v(x)|^{2^*-1}|v(y)|^{2^*-1}}{|x-y|^{N-2}}\,dx\,dy-\int_{\Omega\times \Omega}H(x, y)|v(x)|^{2^*-1}|v(y)|^{2^*-1}\,dx\,dy\\
             %   &\ge \gamma_N\int_{\Omega\times \Omega} \frac{|w(x)|^{2^*-1}|w(y)|^{2^*-1}}{|x-y|^{N-2}}\,dx\,dy+
              %  \gamma_N\int_{\Omega\times \Omega}\frac{|u(x)|^{2^*-1}|u(y)|^{2^*-1}}{|x-y|^{N-2}}\,dx\,dy\\
                %&\quad -\int_{\Omega\times \Omega}H(x, y)|w(x)|^{2^*-1}|w(y)|^{2^*-1}\,dx\,dy-\int_{\Omega\times \Omega}H(x, y) |u(x)|^{2^*-1}|u(y)|^{2^*-1}\,dx\,dy\\
               % &\quad -\int_{S_w\times S_u}H(x, y)|w(x)|^{2^*-1}|u(y)|^{2^*-1}\,dx\,dy-\int_{S_u\times S_w}H(x, y)|u(x)|^{2^*-1}|w(y)|^{2^*-1}\,dx\,dy\\
               % &=\into \phi_w|w|^{2^*-1}\, dx + \into \phi_u|u|^{2^*-1}\, dx - R(w,u)
           % \end{align*}
 and then we conclude.
    \end{proof}

Let $P_k : H^1_0(\Omega) \to H^-$ the orthogonal projection.  In view of Lemma \ref{lemma2} if $m$ is large enough then (see \cite{GR})
$$P_k H_m^-=H^-\qquad \mbox{and}\qquad H_m^- \oplus H^+ =H^1_0(\Omega).$$
Moreover, let us define
$$\mathcal H=\{ h\in C(\bar Q_\e^m, H^1_0(\Omega)): h|_{\partial Q_\e^m}=\operatorname{id}_{\partial Q_\e^m} \}$$
where
%$$Q_m^\e:= \left[\( B(0, R) \cap H_m^-\) \oplus [0, R]\{u_\e\}\right].$$
$$Q_\e^m:=  \{ w\in H_m^-:\|w\|<R \}  \oplus [0, R]\{u_\e\}.$$

Now we verify that the geometrical assumptions of the Linking Theorem hold in our case.
\begin{lemma}\label{le:link}
	Let $\lambda_k \le \lambda < \lambda_{k+1}$. We have that there exist $0<\rho<R$ such that, uniformly for $\e$ sufficiently small,
	\begin{enumerate}[label=(\alph*),ref=\alph*]
		\item \label{link1} $\{u\in H^+ : \|u\|=\rho\}$ and $\partial Q_\e^m$ link;
		\item \label{link2} the functional $I$ is bounded from below by a positive constant on $\{u\in H^+ : \|u\|=\rho\}$;
		\item \label{link3} $\sup_{u\in \partial Q_\e^m}I(u) \le \omega_m $ with $\omega_m \to 0$ as $m\to\infty$.
	\end{enumerate}
\end{lemma}

\begin{proof}
	Property (\ref{link1}) is standard (see \cite[Lemma 1.3]{BR}).\\
	Moreover, by \eqref{phiuu2*-1} of Lemma \ref{Lemmaphi}, we have
	$$I(u) \ge \frac 12\|u\|^2-\frac{\lambda}{2}|u|_2^2 - C\|u\|^{2(2^*-1)}.$$
	Since $\lambda<\lambda_{k+1}$, if $u\in H^+$
	%\todo[inline]{$\lambda_{k+1}=\min\left\{\frac{\|u\|^2}{|u|_2^2}:u\in H^+\setminus\{0\}\right\}$.}
	we find
	$$I(u) \ge C(\|u\|^2-\|u\|^{2(2^*-1)})$$
	for any $u\in H^+$. Therefore, if $\rho$ is small enough, we get  that $$I(u)\ge C>0$$ for all $u\in \{u\in H^+ : \|u\|=\rho\}$ and so (\ref{link2}) is proved.\\
	%{\color{blue}
	%Per Antonio:\\
	%$\partial Q_m^\e= (\{u\in H^-_m:\|u\|=R\}\oplus [0,R]\{u_\e\})\cup \{u\in H^-_m:\|u\|<R\} \cup (\{u\in H^-_m:\|u\|<R\}\oplus \{Ru_\e\})$ (cilindro)}\\
	To prove (\ref{link3}) we observe that,
	by \eqref{ahah}, $I(w)\le \omega_m$.\\
	Moreover, 
	by \eqref{J} and \eqref{stimebn}, we have
	%$$
	%I(R u_\e)
	%\le \frac{NR^2}{N+2} \|u_\e\|^2-\frac{N-2}{N+2}R^{2^*}|u_\e|_{2^*}^{2^*}\\
	%&{\color{red} =\frac{S^{\frac N 2} r^2}{2^*-1} \left( \frac{2^*}{2}-  r^{2^*-2}\right)+r^2 O\(\e^{N-2}\)- r^{2^*}O(\e^N)}\\
	%= \frac{S^{\frac N 2} R^2}{N+2} [ N-(N-2)  R^{2^*-2}]+R^2 O\(\e^{N-2}\)+R^{2^*}O(\e^N)
	%$$
	\begin{align*}
		I(R u_\e)
	&\le \frac{NR^2}{N+2} \|u_\e\|^2-\frac{N-2}{N+2}R^{2^*}|u_\e|_{2^*}^{2^*}\\
	&= \frac{S^{\frac N 2} R^2}{N+2} \left[ N+O(\e^{N-2})-\big(N-2+O(\e^{N})\big)  R^{2^*-2}\right]
	\end{align*}
	and 
	this becomes negative 
	for $R$ large enough and uniformly for $\e$ small.
	Then, by \eqref{ahah} and \eqref{Isplit1}, we have $I(w+Ru_\e)\le \omega_m$.\\
	Finally, since $\max_{0\le r\le R} I(r u_\e) <+\infty$ and
	%, using the equivalence of the norms in $ H^-_m$, 
	for all $w\in H^-_m$ with $\|w\|=R$ we have
	\[
	I(w)\le \frac{N}{N+2}R^2 - C R^{2^*},
	\]
	by \eqref{Isplit1} we conclude that $I(v)\le 0 $ in $\{w\in H^-_m:\|w\|=R\}\oplus [0,R]\{u_\e\}$ if $R$ is sufficiently large.
\end{proof}
Let us set $$c:= \inf_{h\in \mathcal H}\sup_{v\in Q_\e^m} I(h(v)).$$ Since $\operatorname{id}_{\partial Q_\e^m}\in\mathcal H$ we have $$c\le \sup_{v\in Q_\e^m} I(v)$$  and, of course, $c>0$ by \eqref{link1} and \eqref{link2} of Lemma \ref{le:link}.\\
We complete our proofs if we show that for $\e$ small enough
%\\{\color{blue}$\max$}
\begin{equation}\label{aim}
\max_{v\in Q_\e^m} I(v) <\frac{2}{N+2}S^{\frac N 2}.
\end{equation}

This is the aim of the last part of the section. We distinguish two cases:  
	\begin{itemize}
		\item $\lambda\not\in\sigma(-\Delta)$, and $\lambda>\lambda_1$; 
		\item $\lambda\in\sigma(-\Delta)$ and $N\ge 6$.
	\end{itemize}

\subsection{Sign changing solutions: the nonresonance case}

Here we assume that there exist $k\in \mathbb N^*$ and $\theta>0$ such that
\begin{equation}\label{condlambda}
\lambda_k+\theta\le \lambda < \lambda_{k+1}.
\end{equation}
We prove \eqref{aim} under condition \eqref{condlambda}.
\\ Let us choose $m$ so large such that
\begin{equation}\label{rel}
c_k m^{2-N}<\theta
\end{equation}
where $c_k$ is as in Lemma \ref{lemma2}.\\
By contradiction let us assume that for all $\e>0$
\begin{equation*}%\label{aimcontr}
\sup_{v\in Q_\e^m} I(v) \ge\frac{2}{N+2}S^{\frac N 2}.
\end{equation*}
It is easy to see that
%The set $\{ v\in Q_\e^m:I(v)\ge 0\}$ is compact and then 
the supremum is attained and then, for all $\e>0$, there exist $w_\e\in H_m^-$ and $t_\e\ge 0$ such that
\begin{equation}\label{contr}
I(v_\e)
%=I(w_\e+t_\e u_\e)
=\max_{v\in Q_\e^m} I(v)\ge\frac{2}{N+2}S^{\frac N 2},
%\frac 12 \|v_\e\|^2-\frac{\lambda}{2}|v_\e|_2^2 -\frac{1}{2p}|\nabla\phi_{v_\e}|_2^2 \ge\frac{2}{N+2}S^{\frac N 2}
\end{equation}
where $v_\e:=w_\e+t_\e u_\e$.\\
%From 
%%\eqref{contr} 
%the definition of $Q_\e^m$,  %and (\ref{link3}) of Lemma \ref{le:link}
%it follows that the sequences $(t_\e)\subset \mathbb R_+$ and $(w_\e)\subset H_m^-$ are bounded.
%Hence, up to a subsequence, we may assume that
%$$t_\e \to t_0\ge 0,\qquad w_\e\to w_0\in H_m^-, \qquad\hbox{as }\e\to 0,$$ where the convergence of $(w_\e)$ can be viewed in any norm since $H_m^-$ is finite dimensional.\\
%{\color{green}eliminare le parti in rosso in \eqref{eq:estJep} \eqref{eq:Jweps} e \eqref{eq:split}: controllare}\\
Observe that by Lemma \ref{lemma2}, \eqref{condlambda} and \eqref{rel} we have
\begin{equation}\label{eq:Jweps}
I(w_\e)
\le \frac{c_k m^{2-N}-\theta}{2(\lambda_k +c_k m^{2-N})}\| w_\e\|^2-\frac 1{2(2^*-1)}\into \phi_{w_\e}|w_\e|^{2^*-1}\,dx
\le 0,
\end{equation}
which improves the estimate \eqref{ahah}.\\
Moreover we have the following property of the family $(t_\e)$.

\begin{lemma}\label{44}
There exists $C>0$ such that $t_\e\ge C$ for all $\e>0$. In particular we have that there exists a sequence $(\e_n)\subset\R_+^*$ such that $\e_n\to 0$ and $t_{\e_n}\to 1$ as $n\to +\infty$.
\end{lemma}
\begin{proof}By \eqref{Isplit1}, \eqref{contr} and \eqref{eq:Jweps} we deduce that $t_\e\ge C>0$ for all $\e>0$. \\
Since $(t_\e)$ is bounded, there exists $(\e_n)\subset\R_+^*$ such that $\e_n\to 0$ and $t_{\e_n}\to t_0$ as $n\to +\infty$. For simplicity we label $\e=\e_n$ and so $\e\to 0$.\\
By \eqref{bubble} we have that $|t_\e u_\e|_2^2  \to 0$ as $\e\to 0$.
Hence, by \eqref{J} and \eqref{stimebn} we have
$$I(t_\e u_\e) \le  \left( \frac{N}{N+2} t_0^2-\frac{N-2}{N+2} t_0^{2^*}\right)S^{\frac N 2} +o(1).$$
%{\color{blue}
Then, by Lemma \ref{le:almsplit},
    \begin{equation}\label{eq:split}
      I(v_\e)
      \le
      I(w_\e) + I(t_\e u_\e)
\le I(w_\e)+ \varphi(t_0) S^{\frac N 2} +o(1)
    \end{equation}
where $$\varphi(t)=\frac{N}{N+2} t^2-\frac{N-2}{N+2} t^{2^*}.$$
%Observing that for any $\pedfrac{(x,y)}{(y,x)}\in B(0,\frac 1 m)\times \Omega\setminus B(0,\frac 1 m)$ we have from \eqref{propH} that $H(x, y)\le C $, we compute
   % \begin{equation}\label{eq:R}
      %  R(w_\e,t_\e u_\e) \le C|w_\e|_{2^*-1}^{2^*-1}|u_\e|_{2^*-1}^{2^*-1}\le C|w_\e|_{2^*}^{2^*-1}\e^{\frac {N-2}2}
    %\end{equation}
%where in the last line inequality we use the equivalence of norms in $H^-_m$ and the fact that $|u_\e|_{2^*-1}^{2^*-1}=O(\e^{\frac{N-2}2})$.\\}
We note that
$$\max_{t\ge 0} \varphi(t) =\varphi(1)=\frac{2}{N+2}
\quad
\hbox{and}
\quad
\varphi(t)<\frac{2}{N+2}\,\, \forall t\ge 0, t\neq 1.$$
So, taking into account \eqref{eq:Jweps}, if it was $t_0\neq 1$, then for $\e$ small inequality \eqref{eq:split} would contradict \eqref{contr}.
\end{proof}

Moreover we recall the following result.
%{\color{blue}Il lemma corrisponde a \cite[Lemma 5]{GR} e le espressioni della funzione, se non servono, possiamo ometterle. In questo caso io aggiungerei la parte in blu e cancellerei la dimostrazione.}
\begin{lemma}[{\cite[Lemma 5]{GR}}]\label{le:2}
There exists a function $\tau=\tau(\e)$ such that $\tau(\e) \to +\infty$ as $\e\to 0$ and $$\int_\Omega |t_\e u_\e|^2\, dx \ge \tau(\e) \e^{N-2}$$ for $\e$ small.\end{lemma}
%\begin{proof}
%The proof can be made as in \cite[Lemma 5]{GR} and $\tau(\e)=\e^{2-\frac N 2}	$ for $N\ge 5$ and $\tau(\e)=|\log\e|$ for $N=4$.
%\end{proof}
%{\color{blue}Forse ci vogliono delle costanti davanti alle due espressioni di $\tau$.}\\

Hence we can conclude the proof in this case.
\begin{proof}[Proof of (\ref{ii13}) of Theorem \ref{Th2} completed]
First we observe that, by \eqref{J}, \eqref{stimebn}, Lemma \ref{le:2}
\begin{align*}
I(t_\e u_\e)
&\le
\frac{N}{N+2}t_\e^2\| u_\e\|^2
-\frac{\lambda}{2}|t_\e^2 u_\e|_2^2
-\frac{N-2}{N+2}t_\e^{2^*}| u_\e|_{2^*}^{2^*}\\
&\le
\(\frac{N}{N+2}t_\e^2 -\frac{N-2}{N+2}t_\e^{2^*}\) S^{\frac N 2}
+C(1 - \tau(\e)) \e^{N-2}\\
&<
\frac{2}{N+2} S^{\frac N 2}
\end{align*}
%\begin{equation}
%\label{stimasuIteue}
%\begin{split}
%\end{split}
%\end{equation}	
for $\e$ small enough.
%\begin{lemma}\label{le:1}
%	As $\e\to 0$
%	we have $$\frac{N}{N+2} \|t_\e u_\e\|^2-\frac{N-2}{N+2} |t_\e u_\e |_{2^*}^{2^*}\le \frac 2{N+2} S^{\frac N 2}+ O\(\e^{N-2}\).$$
%\end{lemma}
%\begin{proof}
%	By 
%	$$\|t_\e u_\e\|^2 \le S^{\frac N 2} +(t_\e^2-1)  S^{\frac N 2}+ O\(\e^{N-2}\)$$ $$|t_\e u_\e|_{2^*}^{2^*} \ge  S^{\frac N 2}+ (t_\e^{2^*}-1) S^{\frac N 2}+O\(\e^N\)$$
%	as $\e\to 0$. Then
%	\begin{align*}
%	\frac{N}{N+2} \|t_\e u_\e\|^2-\frac{N-2}{N+2} |t_\e u_\e |_{2^*}^{2^*}
%	&\le \(\frac{2^*}{2(2^*-1)}-\frac{1}{2^*-1}\) S^{\frac N 2}\\
%	&\quad+ S^{\frac N 2}\(\frac{2^*}{2(2^*-1)}(t_\e^2-1) -\frac{1}{2^*-1} (t_\e^{2^*}-1)\)+O\(\e^{N-2}\)
%	\end{align*}
%	We observe that $$\max_{x\ge 0}\left(\frac{2^*}{2(2^*-1)}(x^2-1) -\frac{1}{2^*-1} (x^{2^*}-1)\right)=0$$ and the conclusion follows.
%\end{proof}
Then by \eqref{Isplit1} and \eqref{eq:Jweps}  we get
\[
I(v_\e)<\frac{2}{N+2}S^{\frac N 2}
\]
which contradicts \eqref{contr}.\\
Hence, the Linking Theorem can be applied using Lemma \ref{PS} and Lemma \ref{le:link}.
\end{proof}

%\footnote{
%\begin{align*}
%I(v_\e)
%&\le I(w_\e)+ I(t_\e u_\e) - R(w_\e,t_\e u_\e)\\
%&\le
%\frac{2^*}{2(2^*-1)}\|t_\e v_\e\|^2 - \frac{\lambda}{2} |t_\e v_\e|_2^2 - \frac{1}{2^*-1} |t_\e v_\e|_{2^*}^{2^*}\\
%&\le
%\frac{2}{N+2}S^{\frac N 2}+ (c-\tau(\e))\e^{N-2}
%\end{align*}
%and for $\e$ small this contradicts \eqref{contr}.
%}
%\\

\subsection{Sign changing solutions: the resonance case}%$N\ge 6$ and $\lambda\in\sigma(-\Delta)$}
Let $\lambda=\lambda_k$ and  $\mu\in (\lambda_k, \lambda_{k+1})$.
As in the previous section we want to show that \eqref{aim} holds.\\
By contradiction assume that for all $m$ and all $\e>0$ there exists $v_\e^m=w^m_\e+t^{m}_\e u_\e^m \in Q_\e^m$ such that
\begin{equation}\label{aimcontr1}
I(v_\e^m)
%=\frac 12 \|v_\e^m\|^2-\frac{\lambda_k}{2}|v_\e^m|_2^2-\frac{1}{2(2^*-1)}\|\phi_{v_\e^m}\|^2 
\ge \frac{2}{N+2}S^{\frac N 2}
\end{equation}
where $u_\e^m$ is defined in \eqref{uem}.\\
Then, as in Lemma \ref{44}, but using \eqref{ahah} instead of \eqref{eq:Jweps},  the sequence $(t^{m}_\e)$
% and $(w_\e^m)$ 
satisfies again
	\begin{equation}\label{eq:bound}
		t^m_\e \ge C>0
	\end{equation}
uniformly for $m$ large enough and $\e>0$.\\
%{\color{green}questa risulta vera per m grande: dipende da $I(w_\e^m)$}{\color{magenta} e da 4.22},$$\|w_\e^m\|\le C_2.$$
As in \cite[Lemma 6]{GR} one can prove the following result.

\begin{lemma}\label{lemma6}
Let $m\to+\infty$ and assume that $\e=\e(m)=o(\frac 1 m)$; then $$\|u_\e^m\|^2=S^{\frac N 2}+ O[(\e m)^{N-2}],\qquad |u_\e^m|_{2^*}^{2^*}=S^{\frac N 2}+ O[(\e m)^N].$$ Moreover,
$$\int_\Omega |t_\e^m u_\e^m|^2\, dx\ge C \e^2.$$
%{\color{green}if $N\ge 5$ (va omesso:l'abbiamo gi\'a scritto all'inizio della sezione}) then there exists a function $\psi$ such that $\lim_{x\to+\infty} \psi(x)=+\infty$ and $$\int_\Omega | t_\eu_\e^m|^2\, dx\ge \e^{\frac{N(N-2)}{N+2}}\psi(\e^{-1}).$$
\end{lemma}
	\begin{proof}
		We only prove the last estimate: for any $m$ large enough such that \eqref{eq:bound} holds and $\e$ small enough such that $B_\e\subset B_\frac{1}{2m}$, we have 
			\begin{align*}
				\int_\Omega |t_\e^m u_\e^m|^2\, dx&\ge C\int_0^{\e}\frac{\e^{N-2}}{(\e^2+r^2)^{N-2}}r^{N-1}\,dr = C\e^2.
			\end{align*}

	\end{proof}

Hence we can conclude as follows.
\begin{proof}[Proof of Theorem \ref{Th3} completed]
We choose $\e=\e(m)$ in order to deal only with the parameter $m$ and to have a contradiction to \eqref{aimcontr1} for $m$ large enough. We take
\begin{equation*}%\label{epsm}
\e(m)=m^{-\frac{N-2}{N-4}}\log^{-\frac{1}2}(m).
\end{equation*}
Therefore, as $m\to+\infty$ then $\e(m)=o(\frac 1 m)$ and Lemma \ref{lemma6} applies.\\
From now on, we express only the dependence on $m$.\\
By \eqref{J}, Lemma \ref{lemma6}, if $N\ge 5$ and $m$ is large enough we have
\begin{align*}
I(t_m u_m)
&\le
\frac{N}{N+2}t_m^2\| u_m\|^2
-\frac{\lambda_k}{2}|t_m^2 u_m|_2^2
-\frac{N-2}{N+2}t_m^{2^*}| u_m|_{2^*}^{2^*}\\
&\le
\(\frac{N}{N+2}t_m^2 -\frac{N-2}{N+2}t_m^{2^*}\) S^{\frac N 2}
- \big(C_1 - C_2\log^{-\frac{N-4}{2}}(m)\big) m^{-\frac{2(N-2)}{N-4}}\log^{-1}(m)
\\
&\le
\frac{2}{N+2} S^{\frac N 2}
- C m^{-\frac{2(N-2)}{N-4}}\log^{-1}(m).
\end{align*}
%\begin{lemma}\label{lemma8}
%If $m\to+\infty$ then
%$$I(w^m) \le C m^{-\frac{(N+2)(N-2)}{4}}.$$
%\end{lemma}
%
%\begin{proof}
%
%Proceeding as in \eqref{eq:Jweps} and taking into account that $\l=\l_k$ %(a proposito, se nel caso 2 diciamo che $\l\in\sigma(-\Delta)$, dobbiamo scrivere $\l$ come se fosse un autovalore),
%
%\end{proof}
Then, using also \eqref{Isplit1} and \eqref{ahah}, we have
\begin{align*}
I(v^m) &\le I(t_m u^m)+I(w^m)\\
&\le \frac{2}{N+2}S^{\frac N 2}
- C_1 m^{-\frac{2(N-2)}{N-4}}\log^{-1}(m)
+C_2 m^{-\frac{N^2-4}{4}} \\
& =
\frac{2}{N+2}S^{\frac N 2}
- \Big(C_1 \log^{-1}(m) - C_2 m^{-\frac{(N-2)(N^2-2N-16)}{4(N-4)}}\Big)m^{-\frac{2(N-2)}{N-4}} 
\end{align*}
which, for $m$ large enough and $N\ge 6$, contradicts \eqref{aimcontr1}.\\
Hence, the Linking Theorem, through Lemma \ref{PS} and Lemma \ref{le:link}, allows us to conclude.
\end{proof}

\end{document}